\documentclass{amsart}

\usepackage{amsmath, amsthm, amsfonts, amssymb,enumerate,bbm}
\usepackage{mathrsfs,color,soul}
\allowdisplaybreaks

\newtheorem{theorem}{Theorem}[section]
\newtheorem{lemma}[theorem]{Lemma}
\newtheorem{hyp}[theorem]{Hypotheses}
\newtheorem{hypo}[theorem]{Hypothesis}
\newtheorem{prop}[theorem]{Proposition}
\newtheorem{cor}[theorem]{Corollary}

\theoremstyle{definition}

\theoremstyle{remark}
\newtheorem{remark}[theorem]{Remark}

\numberwithin{equation}{section}

\newcommand{\one}{\mathbbm{1}}
\newcommand{\R}{\mathbb{R}}
\newcommand{\Rd}{\mathbb{R}^d}

\newcommand{\N}{\mathbb{N}}

\newcommand{\cA}{\mathcal{A}}
\newcommand{\cT}{\mathcal{T}}

\newcommand{\cL}{\mathcal{L}}

\newcommand{\cG}{\mathcal{G}}

\newcommand{\ph}{\varphi}
\newcommand{\ep}{\varepsilon}

\begin{document}

\title[Strong convergence of solutions to Kolmogorov equations]{Strong convergence of solutions to
nonautonomous Kolmogorov equations}


\author[L. Lorenzi]{Luca Lorenzi}
\address{L.L. \& A.L.: Dipartimento di Matematica e Informatica, Universit\`a degli Studi
di Parma,  Parco Area delle Scienze 53/A, I-43124 Parma, Italy.}
\email{luca.lorenzi@unipr.it, alessandra.lunardi@unipr.it}
\thanks{This work has been supported by the M.I.U.R. Research Project PRIN 2010-2011
``Problemi differenziali di evoluzione: approcci deterministici e stocastici e loro interazioni''. The first
author wishes to thank the Department of Mathematics of the Karlsruhe Institute of Technology for the kind
hospitality during his visit.}

\author[A. Lunardi]{Alessandra Lunardi}

\author[R. Schnaubelt]{Roland Schnaubelt}
\address{R.S.: Department of Mathematics, Karlsruhe Institute of Technology, D-76128 Karlsruhe, Germany.}
\email{schnaubelt@kit.edu}


\date{}

\dedicatory{}

\commby{}

\begin{abstract}
We study  a class of nonautonomous, linear,
parabolic equations with unbounded coefficients on $\R^{d}$
which admit an evolution system of measures.
It is shown that the solutions of these equations converge
to constant functions as $t\to+\infty$. We further establish
the uniqueness of the tight evolution system of measures
and treat the case of converging coefficients.
\end{abstract}


\maketitle


\bibliographystyle{amsplain}



\section{Introduction}
In this paper we investigate the asymptotic behaviour of a class of nonautonomous parabolic
partial differential equations of second order in $\R^d$
with unbounded coefficients.  We establish that the solutions
converge to constant functions as the time  $t$ tends to $+\infty$. These limits exist both locally
uniformly and in $L^p$ spaces with respect to a time-varying family of (invariant) measures.
Such convergence results have been known before only for special cases,
where different more specific methods could be employed, see \cite{ALL, DaPLun07Orn, GeisLun09, LorLunZam10}.

The analysis of nondegenerate elliptic operators with unbounded coefficients goes back to the second
half of last century with the pioneering papers by  Aronson and Besala
\cite{aronson-besala67, aronson-besala2}, Bodanko
\cite{bodanko67}, Feller \cite{feller} and  Krzy\.za\'nski \cite{krzyzanski62-a,krzyzanski62-b}.
The interest of the mathematical community has grown considerably since the nineties because of the many
applications to stochastic analysis, where they appear naturally as Kolmogorov operators of
stochastic partial differential equations, to mathematical finance and also to physics (see e.g. \cite{friedlin}).
Starting from the analysis of autonomous Ornstein--Uhlenbeck equations in \cite{daprato-lunardi},
elliptic operators with unbounded coefficients and the associated Cauchy problems have been studied
both in the space $C_b$ of bounded continuous functions
and in $L^p$ spaces on $\R^d$ and on unbounded domains. In the present paper, we focus on
$\R^d$ for simplicity.

It turned out that the usual $L^p$ spaces with respect to the Lebesgue measure are not appropriate for these
investigations. For instance, no realization of the operator $\cA u=u''-x|x|^{\varepsilon}u'$ in one spatial dimension
generates a $C_0$--semigroup in $L^p(\R)$ if $\varepsilon$ is positive (see \cite{PRS}). This example indicates that
 one needs rather restrictive growth conditions  to develop a theory for elliptic operators with unbounded
 coefficients in $L^p(\R^d)$. The picture changes drastically if the semigroup $T(\cdot)$ associated to $\cA$ on
$C_b(\R^d)$ admits an invariant measure $\mu$ and if one works in the spaces  $L^p(\R^d,\mu)$.
 A probability measure $\mu$ is called invariant if
\begin{eqnarray*}
\int_{\R^d}T(t)f\,d\mu=\int_{\R^d}f\,d\mu
\end{eqnarray*}
for all $f\in C_b(\R^d)$ and $t\ge0$. An invariant measure exists if $\cA$ admits a so-called Lyapunov function,
see Hypothesis~\ref{hyp1}(iii) below,
which is satisfied by large classes of (possibly rapidly growing) coefficients.
We stress that $T(\cdot)$ may not admit an invariant measure;
but if an invariant measure exists, it is unique in our setting.
We refer to e.g.\ \cite{BerLorbook,MPW} for  details on the autonomous case.

If $T(\cdot)$ admits an invariant measure,
it can be extended to a strongly continuous semigroup on
$L^p(\R^d,\mu)$ for each $p\in [1,+\infty)$. The invariant measure $\mu$ also plays an important role
in the analysis of the long-time behaviour of the semigroup $T(\cdot)$. More precisely,
under suitable assumptions the function $T(t)f$ tends, as $t\to +\infty$,
 to the average of $f$ with respect to $\mu$ in $L^p(\R^d,\mu)$
 if $f\in  L^p(\R^d,\mu)$,  and the convergence is locally uniform in $\R^d$
if $f\in C_b(\R^d)$, cf.\ \cite{LorLunZam10}.

In this paper we treat the nonautonomous case, where the coefficients of the elliptic operator also depend
on time $t\ge0$.  The semigroup $T(\cdot)$ of the autonomous case is now  replaced by an evolution operator
$\{G(t,s):t\ge s\ge 0\}$ in $C_b(\R^d)$. Its existence and its main properties
have been established in \cite{DaPLun07Orn} for nonautonomous Ornstein--Uhlenbeck operators and
in \cite{KunLorLun09Non} for the general case.
For $f\in C_b(\R^d)$ and $s\ge0$, the function $G(\cdot,s)f$ is defined as the
unique solution $u$ in  $C_b([s,+\infty)\times\R^d)\cap C^{1,2}((s,+\infty)\times\R^d)$ of the parabolic equation
$D_tu=\mathcal A(t) u$ on $(s,+\infty)\times\Rd$ satisfying  $u(s,\cdot)=f$ in $\Rd$.

Similarly, the concept of invariant measure is replaced by the concept of evolution systems of measures
(as referred to in \cite{DPR}). Such a system is a one-parameter family of probability measures
$\{\mu_t: t\ge0\}$ satisfying
\begin{equation}
\int_{\R^d}G(t,s)f\,d\mu_t=\int_{\R^d}f\,d\mu_s
\label{invariance-1}
\end{equation}
for all $f\in C_b(\R^d)$ and $t\ge s\ge 0$. As in the autonomous case, Lyapunov functions provide a convenient
sufficient condition for the existence of an evolution system of measures, see Hypothesis~\ref{hyp1}(iii).
Under such assumption, the proof of Theorem~5.4 of \cite{KunLorLun09Non} even implies the existence of a
tight evolution system of measures; i.e., for  every $\varepsilon>0$
there exists a radius $R>0$ such that $\mu_t(B_R)\ge 1-\varepsilon$ for all $t\ge 0$. In \cite{GeisLun08}
it was shown that nonautonomous Ornstein--Uhlenbeck evolution operators admit infinitely many  evolution systems of
measures under reasonable assumptions. One can however derive uniqueness within certain classes of evolution
systems, see \cite{ALL, GeisLun08, LorLunZam10} for such results in various cases.

If the evolution operator $G(t,s)$ admits an evolution system of measures, then it can be extended
to a contraction from $L^p(\R^d,\mu_s)$ to $L^p(\R^d,\mu_t)$ for every $t\ge s\ge0$,
Indeed, Proposition \ref{prop:properties}(i) implies that $|G(t,s)f|^p\le G(t,s)(|f|^p)$
for all $f\in C_b(\R^d)$
and $t\ge s\ge 0$. Integrating this inequality with respect to $\mu_t$, we obtain
\begin{equation}
\|G(t,s)f\|_{L^p(\Rd,\mu_t)}^p\le \int_{\Rd}G(t,s)(|f|^p)d\mu_t=\int_{\Rd}|f|^pd\mu_s=\|f\|_{L^p(\Rd,\mu_s)}^p,
\label{eq:contr}
\end{equation}
for all $t>s\ge 0$ and $p\in [1,+\infty)$. Each measure $\mu_r$ is equivalent to the Lebesgue measure $\lambda$
since it has a positive density $\rho(r,\cdot)$ with respect to $\lambda$ by results in \cite{BogKryRoc01},
see p.~2067 there.
But the spaces $L^p(\R^d,\mu_t)$ and $L^p(\R^d,\mu_s)$ differ in general for $t\neq s$. This fact
causes several difficulties in the analysis and, in particular, the standard theory of evolution
operators (e.g.\ in \cite{CL}) can not be applied to the evolution operator in the $L^p$--spaces for $\mu_t$.
As in \cite{ALL, GeisLun08, GeisLun09,LorLunZam10,LorZam09}, we will use
the evolution semigroup $\cT(\cdot)$ associated with $G(t,s)$, which is defined by
\begin{equation}
(\cT(t)h)(s,x)=(G(s,s-t)h(s-t,\cdot))(x), \qquad (s,x)\in\R^{1+d},
\label{evol-semi}
\end{equation}
for $t\ge0$ and $h\in C_b(\R^{1+d})$. Here we extend the given coefficients constantly to $t<0$
to obtain an evolution operator $G(t,s)$ for $t\ge s$ on $\R$, as explained in Remark~\ref{rem:r}.

In the papers \cite{ALL,GeisLun09,LorLunZam10} for several special cases
it was established that $G(t,s)f$ converges to the average $m_s(f):=\int_{\R^d}fd\mu_s$  as $t\to +\infty$.
 For bounded diffusion coefficients and time-periodic coefficients, Corollary~3.8 of \cite{LorLunZam10}
shows that $\|G(t,s)f-m_s(f)\|_{L^p(\Rd,\mu_t)}$ tends to 0 as $t\to+\infty$ for  $f\in L^p(\Rd,\mu_s)$
and $s\in\R$. The proof of this result relies on
the fact that one can employ the evolution semigroup on the \emph{compact} time interval
$[0,T]$, for the period $T$.

The non-periodic case was addressed in \cite{ALL}, but only for  diffusion coefficients $q_{ij}$
which are constant in the spatial variables and under an  additional strict dissipativity assumption
on the drift term (namely that $r_0<0$ in Hypothesis~\ref{hyp1}(iv) below). These extra conditions yield the
exponentially decaying gradient estimate
$|(\nabla_x G(t,s)f)(x)|\le c e^{r_0(t-s)}(G(t,s)|\nabla f|)(x)$ for all
$t>s$, $x\in\Rd$ and $f\in C_b^1(\R^d)$. This decay property is crucial for the proofs in \cite{ALL}.
In turn, it implies the cyclic condition
$D_iq_{jk}+D_jq_{ki}+D_kq_{ij}= 0$ in $\R\times\R^d$ for all $i,j,k\in\{1,\ldots,d\}$
by Theorem~3.1 in \cite{luciana}, which explains the restriction to space independent
diffusion coefficients $q_{ij}$ in \cite{ALL}.  On the other hand,
Corollary~5.4 of \cite{ALL} even establishes the exponential decay of $\|G(t,s)f-m_s(f)\|_{L^p(\Rd,\mu_t)}$
with rate $r_0<0$.
To the best of our knowledge, this is the only available result on
the long-time behaviour of the function $\|G(t,s)f-m_s(f)\|_{L^p(\Rd,\mu_t)}$ for non-periodic coefficients
(besides \cite{GeisLun09} for the special case of Ornstein--Uhlenbeck operators).

For non-periodic coefficients, our main result Theorem~\ref{thm:main} shows that
  $\|G(t,s)f-m_s(f)\|_{L^p(\Rd,\mu_t)}$ tends to 0 as $t\to+\infty$ if  $f\in L^p(\R^d,\mu_s)$
and that $G(t,s)f$ converges to  $m_s(f)$ locally uniformly if $f\in C(\R^d)$
vanishes at infinity, where $s\ge0$ and $p\in [1,+\infty)$. This theorem then implies the uniqueness of  tight evolution systems
of measures. Compared to \cite{ALL}, we allow for space dependent and possibly
unbounded diffusion coefficients and we do not need the strict dissipativity assumption $r_0<0$ in
Hypothesis~\ref{hyp1}(iv). To use certain estimates on Green's functions, we require additional bounds on the
coefficients which are  global in time but only local in space, see Hypothesis~\ref{hyp1}(i). 

 As in  \cite{GeisLun09,LorLunZam10}, our approach relies on the decay to 0 of $\nabla_x\cT(t)h$ as
 $t\to+\infty$  in $L^p(\R^{1+d},\nu)$ for all $h\in L^p(\R^{1+d},\nu)$, where
 $\nu$ is defined by
\begin{equation}
\nu(A\times B)=\int_A\mu_s(B)ds,
\label{nu}
\end{equation}
on the product of a Borel set $A\subset \R$ and a Borel set $B\subset \Rd$, and canonically extended to
the $\sigma$-algebra of all the Borel sets of $\R^{1+d}$, see Proposition~\ref{prop:gradT}.
This decay is proved by means of a ``carr\'e du champs'' type inequality for the generator of $\cT(\cdot)$,
which we recall in Proposition~\ref{prop:properties}.
To exploit the decay in $L^p(\R^{1+d},\nu)$, we need lower bounds on the
density of $\mu_t$ which are local in space, but uniform in time. We show such estimates in Lemma~\ref{lem:rho}
using known lower bounds of Green's functions solving the Dirichlet problem on a ball, \cite{Ar}.
Still it is rather delicate to pass from the strong convergence of $\nabla_x\cT(t)$ in $L^p(\R^{1+d},\nu)$
to that of $G(t,s)$ in the proof of Theorem~\ref{thm:main}.

As we have already noticed, the spaces $L^p(\R^d,\mu_t)$ differ from each other. If the coefficients
of the {operators $\cA(t)$ converge as $t\to +\infty$, we establish that the solution $G(t,s)f$  tends to the mean
$m_s(f)$  as $t\to +\infty$ in $L^p(\R^d,\mu_{\infty})$ for all $f\in C_b(\R^d)$ (which is
dense in $L^p(\R^d,\mu_s)$), $s\ge 0$ and $p\in [1,+\infty)$, see Theorem~\ref{thm:conv}. Here
$\mu_{\infty}$ is the invariant measure of the semigroup associated to the limiting autonomous operator $\cA_{\infty}$.
The main step in the proof is the convergence result of Proposition~\ref{prop:conv} for the densities
of the invariant measures, where we use the regularity properties of these densities proved in \cite{BogKryRoc01}.
In Section~\ref{example} we exhibit a class of operators that satisfy all our assumptions.

\subsection*{Notation}
We consider the usual spaces $C^{k+\alpha}(\Omega)$ when $\Omega$ is an open
set or the closure of an open set, $k\in\N\cup\{0\}$ and $\alpha\in [0,1)$.
We use the subscript ``$b$'' (resp., ``$c$'') for the subspaces of the above spaces consisting of functions
which are bounded together with all their derivatives up to the order $k$}(resp., are compactly supported).
We also consider the spaces
$C^{1,2}(J\times \Omega)$ and $C^{k+\alpha/2,2k+\alpha}(J\times \Omega)$ for an interval $J$, and use the subscripts
``$b$'' and ``$c$'' with the same meaning as above.
For $\alpha\in (0,1)$ the subscript ``loc'' means that the derivatives of
order $k$ are $\alpha$-H\"older continuous in each compact set contained in $\Omega$ or $J \times \Rd$.

For a Borel measure $\mu$  on $\Omega$ and  $p\in[1,+\infty)$, we denote by $L^p(\Omega,\mu)$ the
usual Lebesgue space (omitting $\mu$ if it is the Lebesgue measure). For an open set
$\Omega\subset\R^d$ and $k\in\N$,  the standard Sobolev space
with respect to the Lebesgue measure is denoted by $W^{k,p}(\Omega)$. Similarly, $W_p^{1,2}(J \times \Omega)$
is the usual parabolic Sobolev space with respect to the Lebesgue measure for an interval $J$.

Given a family of measures $\{\mu_t: t\ge0\}$, we denote by $m_t(f)$ the average
of the function $f$ with respect to the measure $\mu_t$.
Finally,  $B_R$ designates the open ball centered at 0 with radius $R$ and $\R_+:=[0,+\infty)$.

\section{Assumptions and background material}
\label{sect:2}
We deal with differential operators $\cA(t)$, $t\ge 0$, defined  on smooth functions $\ph:\Rd\to\R$ by
\begin{eqnarray*}
\mathcal A(t)\ph =\sum_{i,j=1}^d q_{ij}(t,\cdot) D_{ij}\ph + \sum_{i=1}^d b_{i}(t,\cdot)D_{i}\ph
={\rm Tr}(Q(t,\cdot)D^2_x\ph)+\langle b(t,\cdot),\nabla_x\ph\rangle,
\end{eqnarray*}
under the following conditions that are always assumed throughout the paper.

\begin{hyp}\label{hyp1}
\begin{enumerate}[\rm (i)]
\item
For some $\alpha\in(0,1)$ and every $i,j\in\{1,\ldots,d\}$
$q_{ij}$, $b_i$ belong to $C^{\alpha/2,1+\alpha}_{\rm loc}(\R_+\times \R^d)$. Moreover,
$q_{ij}\in C_b(\R_+\times B_R)$ and
$D_kq_{ij},b_j\in C_b(\R_+;L^p(B_R))$ for all $i,j,k\in\{1,\ldots,d\}$, all $R>0$ and some $p>d+2$.
\item
The matrix $Q(t,x)$ is symmetric and $\langle Q(t,x)\xi,\xi\rangle\geq \eta(t,x)|\xi|^2$ for
all $t\ge0$, $x,\xi\in\Rd$ and a function $\eta:\R_+\times \R^d\to\R$ with $\inf_{\R_+\times \R^d}\eta=:\eta_0>0$.
\item
There exist a function $0<V\in C^2(\R^d)$ and constants $a\ge0$,
$\kappa>0$ such that $V(x)$ tends to $+\infty$ as $|x|\to +\infty$ and
$(\mathcal{A}(t)V)(x)\leq a-\kappa V(x)$ for all $(t,x)\in \R_+\times\R^d$.
\item
There exist  constants $c_0\ge0$ and $r_0\in\R$ such that
$|\nabla_x Q(t,x)|\le c_0\eta(t,x)$ and $\langle \nabla_x b(t,x)\xi,\xi\rangle\leq r_0\,|\xi|^2$
for all $t\ge0$ and $x,\xi\in\Rd$.
\item
There exists a constant $c>0$ such that either $|Q(t,x)|\le c(1+|x|)V(x)$ and
$\langle b(t,x),x\rangle \le c(1+|x|^2)V(x)$
for all $(t,x)\in \R_+\times \R^d$, or $|Q(t,x)|\le c$ for all  $(t,x)\in  \R_+\times \R^d$.
\end{enumerate}
\end{hyp}

Except for the second part in (i), assumptions (i)--(iii) are needed to construct the evolution operator
and the evolution system of measures $\{\mu_t: t\ge 0\}$. Condition (iv) leads to the gradient estimate \eqref{eq:grad}.
The  second part of (i) is needed to obtain uniform lower bounds of the density of the measures,
see Lemma~\ref{lem:rho}. On the last condition we comment in Remark~\ref{rem:v}.

In the next proposition we collect several basic properties of the evolution operator $G(t,s)$.
\begin{prop}
\label{prop-cont-1}
The following properties are satisfied.
\begin{enumerate}[\rm (i)]
\item
Let $D=\{(t,s,x,y)\in\R_+\times\R_+\times \Rd\times \Rd: t>s\}$.  Then there exists  a Green's function
$g:D\to (0,+\infty)$ such that
 \begin{equation}\label{eq:green}
 G(t,s)f =\int_{\Rd} g(t,s,\cdot,y)f(y)\,dy
 \end{equation}
in $\Rd$ for $f\in C_b(\Rd)$ and $t>s\ge 0$. For every $t>s\ge 0$ and  $x\in\Rd$, the function $g(t,s,x,\cdot)$
belongs to $L^1(\R^d)$ and $\|g(t,s,x,\cdot)\|_{L^1(\R^d)}=1$. Each operator $G(t,s)$ is a contraction
on $C_b(\R^d)$ and $G(t,s)\one =\one$.
\item
For every $f\in C_c(\Rd)$ and $t>0$, the function $s\mapsto G(t,s)f$ is  continuous from $[0,t]$
to $C_b(\Rd)$. If $f\in C_c^2(\R^d)$, then for every $(t,x)\in(0,+\infty)\times \R^{d}$
the function  $(G(t,\cdot)f)(x)$ is differentiable in $[0,t]$ and $(D_sG(t,s)f)(x)=-(G(t,s)\cA(s)f)(x)$.
\item
There exists a constant $C_1>0$ such that for all $p\ge 2$ and $s\in\R_+$
\begin{align}
&|\nabla_xG(s+t,s)f|^p\le C_1^p(t^{-p/2}\vee 1)G(s+t,s)(|f|^p),\qquad\,t>0,\;\,f\in C_b(\R^d).
\label{eq:grad}
\end{align}
\end{enumerate}
\end{prop}

\begin{proof}
Statement (i) and (ii) come from Proposition 2.4 (and its proof) and Lemma 3.1 of \cite{KunLorLun09Non}.
Also statement (iii) is a consequence of the results in \cite{KunLorLun09Non} although
it was not explicitly stated there. To prove it, for every $n\in\N$ and $s\ge 0$
we denote by $u_n$ the unique classical solution to Neumann--Cauchy problem
\begin{align*}
D_tu_n(t,x)&=\cA(t)u_n(t,x), \qquad (t,x)\in (s,+\infty)\times B_n,\\
D_\nu u_n(t,x)&=0, \qquad (t,x)\in (s,+\infty)\times\partial B_n,\\
u_n(s,t)&=f(x), \qquad x\in B_n.
\end{align*}
Let $w_n$ solve the same boundary value problem with initial condition $w_n(s,\cdot)=f^2$.
In the proof of Theorem~4.1 of \cite{KunLorLun09Non}  it is shown that the function
\begin{eqnarray*}
(t,x)\mapsto z_n(t,x)=(u_n(t,x))^2+C^{-2}_1(t-s)|\nabla_xu_n(t,x)|^2
\end{eqnarray*}
satisfies the inequality $D_tz_n-\cA z_n\le 0$ in $(s,s+1]\times B_n$, $D_\nu z_n\le 0$ on
$(s,s+1]\times\partial B_n$ and $z_n(s,\cdot)=f^2$ in $B_n$ for each $n\in\N$.
The constant $C_1$ only depends on $\eta_0$, $d$, $c_0$ and $r_0$ from Hypothesis~\ref{hyp1}.
The classical maximum principle now implies that $z_n\le w_n$ in $(s,s+1]\times B_n$.

By Remark~2.3 of \cite{KunLorLun09Non}, the functions $u_n$ and $w_n$ converge to
$G(\cdot,s)f$ and $G(\cdot,s)f^2$, respectively, in $C^{1,2}((s,s+R)\times B_R)$ for every $R>0$.
Taking the limit $n\to +\infty$, the inequality  $z_n\le w_n$ thus yields
the formula \eqref{eq:grad} with $p=2$ for all
$s\ge0$ and $t\in (0,1]$. Let $p>2$. Using \eqref{eq:grad} with $p=2$, H\"older's inequality and
$\|g(t,s,x,\cdot)\|_{L^1(\R^d)}=1$ from (i), we derive
\begin{align*}
|\nabla_xG(s+t,s)f|^p&=(|\nabla_xG(s+t,s)f|^2)^{p/2}\le (C_1^2t^{-1}G(s+t,s)(|f|^2))^{p/2}\\
&\le C_1^pt^{-p/2}G(s+t,s)(|f|^p)
\end{align*}
for all $s\ge 0$, $t\in (0,1]$  and $f\in C_b(\R^d)$. To extend this estimate  to  $t>1$, one finally uses
the evolution law and splits $\nabla_xG(s+t,s)=\nabla_xG(s+t,s+t-1)G(s+t-1,s)f$.
\end{proof}

\begin{remark}\label{rem:r}
Setting  $\cA(t):=\cA(0)$ for $t< 0$, we extend the coefficients $q_{ij}$ and $b_i$ to $t\in\R$ in such a way
that Hypotheses~\ref{hyp1} hold with $\R_+$ replaced by $\R$. Hence, Proposition~\ref{prop-cont-1}
is valid on $\R$ instead of $\R_+$, and for  $t\in\R$ and $(-\infty,t]$ instead of
$[0,t]$ in part (ii). The extended evolution operator is
also denoted by $G(t,s)$, for $t\ge s$ in  $\R$.
Moreover, we set $\mu_s:= G(0,s)^*\mu_0$ for $s<0$, where $G(0,s)^*$ is the adjoint of the
operator $G(0,s)$ in $C_b(\R^d)$. Using formula \eqref{eq:green} to extend $G(0,s)$ to characteristic
functions, it is easy to see that $\mu_s$ is a probability measure for every $s<0$. The set $\{\mu_t: t\in \R\}$ is an evolution
system of measures for $G(t,s)$ on $\R$, see
the proof of Theorem~5.4 of \cite{KunLorLun09Non}.
\end{remark}

We now recall the properties of the evolution semigroup $\cT(\cdot)$ (see \eqref{evol-semi}) and the measure $\nu$ that we use in
this paper.  To define it, we use the evolution operator and the evolution system of measures on  $\R$ from the above remark.

\begin{prop}
\label{prop:properties}
Let $p\in [1,+\infty)$.  The following properties are satisfied.
\begin{enumerate}[\rm (i)]
\item
The measure $\nu$ defined in \eqref{nu} is infinitesimally invariant for $\cT(\cdot)$; i.e.,
\begin{eqnarray*}
\int_{\R^d}(\cA(\cdot)h-D_th)\,d\nu=0\qquad \text{for all \ } h\in C^{\infty}_c(\R^{1+d}).
\end{eqnarray*}
Moreover, the restriction to $C_c(\R^{1+d})$ of the evolution semigroup $\cT(\cdot)$ may be extended to a strongly
continuous contraction semigroup $\cT_p(\cdot)$ in $L^p(\R^{1+d},\nu)$. Its generator is denoted by $\cG_p$.
\item
For any $u\in D(\cG_2)$ the following  {\it ``carr\'e du champs''} type inequality holds true:
\begin{equation}
 \label{integration<Hu}
 \eta_0\! \int_{\R^{1+d}}\!|\nabla_x u|^2\,d\nu \leq
 \int_{\R^{1+d}}\!|Q^{1/2}\nabla_x u|^2\,d\nu \leq  -\!\! \int_{\R^{1+d}}\! u \cG_2 u \,d\nu.
 \end{equation}
\end{enumerate}
\end{prop}

\begin{proof}
We refer the reader to Lemma~6.3(ii) of \cite{KunLorLun09Non} and Theorem~2.1 of \cite{LorZam09}
for part (i) and to  Corollary~2.16 of \cite{LorLunZam10}
for part (ii). The results in \cite{LorLunZam10} are only shown for the case of time-periodic coefficients
with slightly different assumptions from our Hypotheses~\ref{hyp1}. We thus sketch the proof of (ii).

We have to replace the space $D(G_{\infty})$ used in \cite{LorLunZam10} by the space ${\mathcal D}$
of all $u\in C_b(\R^{1+d})$ belonging to $W^{1,2}_p((-R,R)\times B_R)$ for all $R>0$
and $1\le p<+\infty$ such that  $\cG u:={\mathcal A}(\cdot)u-D_tu$ is contained in $C_b(\R^{1+d})$ and
${\rm supp}(u)\subset [-M,M]\times\R^d$ for some $M>0$. The  generator $\cG_2$ is the closure
of the operator $\cG$ defined on ${\mathcal D}$ by Theorem~2.1 of \cite{LorZam09}.
Proposition~2.5 of \cite{LorZam09}  yields
\begin{equation}
u= \int_0^{+\infty} e^{-t} \cT(t)(u-\cG u)\,dt
\label{star-star}
\end{equation}
for all $u\in {\mathcal D}$. The gradient estimate \eqref{eq:grad} for $G(t,s)$ directly implies
the inequality
\begin{equation}\label{est:gradT}
\|\nabla_x\cT(t)h\|_{\infty}\le C_1(t^{-\frac{1}{2}}\vee 1)\|h\|_{\infty}
\end{equation}
for  $h\in C_b(\R^{1+d})$ and $t>0$. As in Proposition~2.14 of \cite{LorLunZam10}, using \eqref{star-star}
we then infer that ${\mathcal D}\subseteq C^{0,1}_b(\R^{1+d})$ and
$\|\nabla_x u\|_\infty\le \tilde c\, (\|u\|_\infty +\|\cG u\|_\infty)$ for $u\in {\mathcal D}$.
Formula \eqref{integration<Hu} can now be shown analogously as
Proposition~2.15 and Corollary~2.16 in \cite{LorLunZam10}, where the first inequality
in \eqref{integration<Hu} follows from Hypothesis~\ref{hyp1}(ii).
\end{proof}

\begin{remark} \label{rem:v}
Hypothesis \ref{hyp1}(v) is crucial to prove the inequality \eqref{integration<Hu}, and this is the only part
of the paper where we use it. Typically, one takes as a Lyapunov function $V(x)= 1+|x|^{2n}$ for
$x\in\R^d$ and some $n\in\N$ or $V(x)=e^{\delta|x|^\beta}$ for  $x\in\R^d$ and some $\beta,\delta>0$,
so that Hypothesis \ref{hyp1}(v) is rather mild. See also the example in Section~\ref{example}.
\end{remark}

In the time periodic case, the next result was shown in Proposition~3.4 of \cite{LorLunZam10} for $p=2$
extending \cite{daprato-goldys} in the autonomous case. In this paper we need it for $p>d$. The proof in
our case follows the same lines as \cite{LorLunZam10}. Nevertheless, since Proposition~\ref{prop:gradT}
is crucial for all our analysis,  we provide a proof for the reader's convenience.

\begin{prop}\label{prop:gradT}
For all $p\in[2,+\infty)$  and $h\in L^p(\R^{1+d},\nu)$ we have
\begin{equation}
\lim_{t\to +\infty} \|\,|\nabla_x\cT_p(t)h|\, \|_{L^p(\R^{1+d},\nu)} =0.
\label{gradient-p}
\end{equation}
\end{prop}

\begin{proof}
The estimate \eqref{eq:grad} implies  that $|\nabla_x\cT(t)h|^p\le C_1^p(t^{-p/2}\vee 1)\cT(t)|h|^p$ on $\R^{1+d}$
 for all  $h\in C_c(\R^{1+d})$, $t>0$, and $p\in [2,+\infty)$. We now integrate this inequality on $\R^{1+d}$
 with respect to the measure $\nu$ and use the density of $C_c(\R^{1+d})$ in $L^p(\R^{1+d},\nu)$.
 It follows that $\nabla_x\cT_p(t)h$ is contained in $L^p(\R^{1+d},\nu)^d$ and
\begin{equation}\label{est:gradTp}
\|\,|\nabla_x\cT_p(t)h|\,\|_{L^p(\R^{1+d})} \le C_1(t^{-1/2}\vee 1)\|h\|_{L^p(\R^{1+d},\nu)}
\end{equation}
for all $h\in L^p(\R^{1+d},\nu)$ and $t>0$, where we use the contractivity of $\cT(t)$ in $L ^1(\R^{1+d},\nu)$.
Combined with the H\"older inequality, this estimate and
\eqref{est:gradT} yield
\begin{align}
\|\,|\nabla_x\cT_p(t)h|\,\|_{L^p(\R^{1+d},\nu)}  &\le\|\,|\nabla_x\cT(t)h|\,\|_{\infty}^{1-\frac{2}{p}}
                                       \,\|\,|\nabla_x\cT_2(t)h|\,\|_{L^2(\R^{1+d},\nu)}^{\frac{2}{p}}\notag\\
&\le C_1\|h\|_{\infty}\|\,|\nabla_x\cT_2(t)h|\,\|_{L^2(\R^{1+d},\nu)}^{\frac{2}{p}}
\label{3star}
\end{align}
for all $t>0$ and $h\in C_c(\R^{1+d})$.  In view of \eqref{est:gradTp} and \eqref{3star}, we thus have to show
\eqref{gradient-p} only for $p=2$ since $C_c(\R^{1+d})$ is dense in $L^p(\R^{1+d},\nu)$. Similarly,
it suffices to prove  \eqref{gradient-p} with $p=2$ for functions in the dense subset $D(\cG_2^2)$
of $L^2(\R^{1+d},\nu)$.

Take $u \in D(\cG_2^2)$. Then, the function $t\mapsto\|{\mathcal T}_2(t)u \|_{L^2(\R^{1+d},\nu)}^2$
is differentiable on $[0,+\infty)$ with derivative
$2\langle {\mathcal T}_2(\cdot)u, \cG_2{\mathcal T}_2(\cdot)u\rangle_{L^2(\R^{1+d}, \nu)}$.
 From  \eqref{integration<Hu}, we then deduce
\begin{align*}
2\eta_0 \int_{0}^{t}\int _{\R^{1+d}}|\nabla _{x}{\mathcal T}_2(s)u|^2\,d\nu \,ds
&\le - \int_{0}^{t} 2\langle {\mathcal T}_2(s)u, \cG_2{\mathcal T}_2(s)u\rangle_{L^2(\R^{1+d}, \nu)}\,ds\\
& = \|u\|_{L^2(\R^{1+d},\nu)}^2-\|\cT(t)u\|_{L^2(\R^{1+d},\nu)}^2\le \|u\|_{L^2(\R^{1+d},\nu)}^2
\end{align*}
for $t\ge0$; i.e., the map $\chi _{u} := \|\,| \nabla_x{\mathcal T}_2(\cdot)u|\,\|_{L^2(\R^{1+d},\nu)}^2$
belongs to $L^1(0,+\infty)$. The estimate \eqref{integration<Hu} also implies that
the gradient $\nabla_x: D(\cG_2)\to L^2(\R^{1+d},\nu)^d$ is continuous.
Since $u \in D(\cG_2^2)$, the function $\chi_u$ is thus differentiable and
\begin{align*}
|\chi_u'| &= 2\,\Big |\int _{\R^{1+d}}\langle\nabla_x{\mathcal T}_2(\cdot)u,
 \nabla_x{\mathcal T}_2(\cdot)\cG_2 u \rangle \, d\nu\Big|\\
& \leq 2\,\|\,|\nabla_x\cT_2(\cdot)u|\,\|_{L^2(\R^{1+d},\nu)}\,
     \|\,|\nabla_x\cT_2(\cdot)\cG_2 u|\,\|_{L^2(\R^{1+d},\nu)}
 \leq\chi_{u}+\chi _{\cG_2 u}
\end{align*}
in $[0,+\infty)$. Using $\cG_2 u \in D(\cG_2)$ once more, we conclude that
also the derivative $\chi_u'$ belongs to $L^1(0,+\infty)$ and so $\chi_u(s)$ vanishes as $s\to +\infty$.
\end{proof}

We conclude this section by a simple convergence lemma for tight sequences of probability
measures.
\begin{lemma}\label{lem:conv}
Let $(\tilde\mu_n)$ be a tight sequence of probability measures in $\Rd$ and $(g_n)\subset C_b(\Rd)$ be
a bounded sequence. The following assertions hold.
\begin{enumerate}[\rm (i)]
\item
If $g_n$ tends to zero locally uniformly in $\Rd$ as $n\to+\infty$, then
$\int_{\Rd} g_n\,d\tilde\mu_n$ vanishes as $n\to+\infty$.
\item
If $g_n$ tends to some $g\in C_b(\Rd)$ locally uniformly in $\Rd$ and $\tilde\mu_n$ converges
weakly$^*$ to  a probability measure $\tilde\mu$ in $\Rd$ as $n\to+\infty$, then
$\int_{\Rd} g_n\,d\tilde\mu_n$ tends to $\int_{\Rd} g\,d\tilde\mu$ as $n\to+\infty$.
 \end{enumerate}
\end{lemma}
\begin{proof}
We only show property (ii), as the first assertion can be treated similarly. By assumption,
$M:=\sup_{n\in\N}\{\|g_n\|_{\infty},\,\|g\|_{\infty}\}<+\infty$ and for each $\varepsilon>0$ there exists a radius
$r>0$ such that $\tilde\mu_n(\Rd\setminus B_r)\le \varepsilon$. We can thus estimate
\begin{align*}
\Big|\int_{\Rd}  g_n\,d\tilde\mu_n\! - \! \int_{\Rd} g\,d\tilde\mu\Big|
&\le\! \int_{B_r}\!\! |g_n\!-\!g|\,d\tilde\mu_n+\int_{\Rd\setminus B_r}\! |g_n\!-\!g|\,d\tilde\mu_n\!
   +\!\Big|\int_{\Rd}\! g\,d\tilde\mu_n\!-\!\int_{\Rd}\! g\,d\tilde\mu\Big|\\
& \le \sup_{x\in B_r} |g_n(x)\!-\!g(x)|+ 2M\varepsilon
   +  \Big|\int_{\Rd} g\,d\tilde\mu_n\!-\!\int_{\Rd} g\,d\tilde\mu\Big|.
\end{align*}
As $n\to+\infty$, the sum  in last line tends to $2M\varepsilon$, and (ii) follows.
\end{proof}

\section{Asymptotic behaviour of $G(t,s)$}
\label{sect:3}
Throughout this section, $\{\mu_t: t\ge 0\}$ is any tight evolution system of measures for $G(t,s)$,
extended to the whole $\R$ as in Remark \ref{rem:r}. We recall that by Theorem 5.4 in \cite{KunLorLun09Non} a tight evolution
system of measures for $G(t,s)$ does exist.
Corollary~3.2 of \cite{BogKryRoc01} yields that there exists a positive function
$\rho:\R^{1+d}\to\R$ such that $\rho(t,\cdot)$ is the density of $\mu_t$ with respect to the Lebesgue measure
for every $t\in\R$. In Corollary~\ref{cor:unique} we will see that actually there exists only
one tight evolution system of measures for $G(t,s)$.

To begin with, we use Hypothesis~\ref{hyp1}(i) to prove a lower bound on the densities $\rho(t,\cdot)$,
 which is crucial in our analysis.

\begin{lemma}\label{lem:rho}
For each $k\in\N$ there exists a number
$\delta_k>0$ such that $\rho(\tau,x)\ge\delta_k$ for all $\tau\ge0$ and $|x|\le k$.
\end{lemma}
\begin{proof}
Let $D_k=\{ (t,s,x,y)\in \R_+\times\R_+\times \overline{B_k}\times \overline{B_k}: t>s\}$ for every $k\in\N$.
By  $g_k:D_k\to [0,+\infty)$ we denote the Green's function of the parabolic problem
\begin{align*}
D_tu(t,x)&=\cA(t)u(t,x), \qquad (t,x)\in (s,+\infty)\times B_k,\\
u(t,x)&=0, \qquad (t,x)\in (s,+\infty)\times\partial B_k,\\
u(s,x)&=f(x), \qquad x\in B_k,
\end{align*}
as constructed in
Theorem~3.16 of \cite{Friedman} and its corollaries.  The proof of
Proposition~2.4 in \cite{KunLorLun09Non} yields $g\ge g_k$ on $D_k$ for each $k\in\N$,
where $g$ is Green's function in Proposition \ref{prop-cont-1}(i). Since the family $\{\mu_t: t\ge 0\}$ is
tight, there is a radius $k_0\in \N$ such that
$\mu_t(\overline{B_{k_0}})\ge 1/2$ for all $t\ge 0$. Throughout the proof, the integer
$k\ge k_0$ is arbitrary, but fixed. We claim that there exists a number $\delta_k>0$ such that
\begin{equation}\label{eq:delta-k}
g_{k+2}(\tau+1,\tau,x,y)\ge 2\delta_k\qquad  \text{for all \ \ } \tau\ge0, \ x,y\in \overline{B_k}.
\end{equation}

To prove the claim, we rewrite the operators ${\mathcal A}(t)$ in divergence form and apply Theorem~9(iii)
in \cite{Ar} with  $\Omega'=B_{k+1}$, $\Omega=B_{k+2}$ and $T=8$ to the operators
$\cL(t)=D_t-{\rm div}(\tilde Q(t+\tau,\cdot)\nabla_x)-\langle \tilde b(t+\tau),\nabla_x\rangle$ on
$(0,1]\times B_{k+2}$ for $\tau\ge0$. Here the  coefficients $\tilde q_{ij}=\tilde q_{ji}$ belong to $C_b(\R_+ \times \R^d)$
and satisfy $\tilde q_{ij}=q_{ij}$ on $\R_+\times B_{k+2}$ as well as   $\langle \tilde Q(t,x)\xi,\xi\rangle\ge\eta_0/2$
for all $i,j\in\{1,\ldots,d\}$, $(t,x)\in \R_+ \times \R^{d}$ and $\xi\in\partial B_1$.
The drift coefficients $\tilde b_i$ are continuous extensions of $b_i-\sum_{j=1}^dD_iq_{ij}$ to $\R^{1+d}$, such that on
$\tilde b_i=b_i-\sum_{j=1}^dD_iq_{ij}$ on  $\R_+\times B_{k+2}$ and $\tilde b_i=0$ on $\R_+\times \R\setminus B_{k+3}$ for $k\ge k_0$ and $i,j\in\{1,\ldots,d\}$.
By the uniqueness statement in Theorem~6 of \cite{Ar}, the map
$g_{k+2}(\cdot+\tau,\tau,\cdot,\cdot)$ is the Green's function of $\cL(t)$
on $(0,1]\times B_{k+2}$. Theorem~9(iii) of \cite{Ar} now implies that
\begin{equation}
g_{k+2}(t+\tau,\tau,x,y)\ge C_1t^{-d/2}\exp(-C_2t^{-1}|x-y|^2)
\label{eq:delta-kk}
\end{equation}
for all $x,y\in B_{k+1}$,   $t\in (0,\min\{8,(d(y,\partial B_{k+1}))^2\}]$ and $\tau\ge0$.
The constants $C_1$ and $C_2$ depend on $\eta_0$, $\sup_{t\ge0}\|q_{ij}(t,\cdot)\|_{L^\infty(B_{k+2})}$,
$\sup_{t\ge0}\|b_j(t,\cdot)\|_{L^p(B_{k+2})}$
and on $\sup_{t\ge0}\|D_i q_{ij}(t,\cdot)\|_{L^p(B_{k+2})}$ for all $i,j\in\{1,\ldots,d\}$ and some $p>d$.
(Note that these suprema are finite due to Hypotheses \ref{hyp1}(i).)
If $y\in B_k$, then $d(y,\partial B_{k+1})\ge 1$. Hence, we can take $t=1$ in \eqref{eq:delta-kk},
and \eqref{eq:delta-k} follows.

We can now complete the proof.  Take a  Borel set
$B\subset \overline{B_k}$ and some $\tau\ge0$.
 From \eqref{invariance-1}, \eqref{eq:green} and \eqref{eq:delta-k}, we deduce
\begin{align*}
\int_{B} \rho(\tau,x)\,dx
  &= \int_{\Rd} \int_{B} g( \tau+1,\tau,x,y)\rho(\tau+1,x)\, dy\, dx\\
 &\ge \int_{\overline{B_k}} \int_{B} g_{k+2}( \tau+1,\tau,x,y)\rho(\tau+1,x)\, dy\, dx\\
 &\ge  2\delta_k\lambda(B) \int_{\overline{B_k}} \rho(\tau+1,x)\, dx=2\delta_k\lambda(B)\mu_{\tau+1}(B_k)
  \ge \delta_k\lambda(B),
\end{align*}
where  $\lambda$ is the Lebesgue measure. This lower bound yields the assertion.
\end{proof}

We now establish our main result on the convergence of $G(t,s)$.

\begin{theorem} \label{thm:main}
Let $s\ge0$, $p\in[1,+\infty)$, and
$\{\mu_t:t\ge 0\}$ be a tight evolution system of measures for $G(t,s)$. The following assertions
are true.
\begin{enumerate}[\rm (i)]
\item
$\|G(t,s)f-m_s(f)\|_{L^p(\Rd,\mu_t)}$
tends to $0$  as $t\to +\infty$ for each $f\in L^p(\Rd,\mu_s)$.
\item
For each $f\in C_b(\Rd)$, $G(t,s)f$ tends to $m_s(f)$ locally uniformly in $\Rd$ as $t\to +\infty$.
\end{enumerate}
\end{theorem}

\begin{proof}
(i) First of all, we observe that it suffices to prove the assertion for $s\in \R_+\setminus {\mathcal N}$, where
${\mathcal N}$ is a null set. Indeed, if $s\in {\mathcal N}$, we fix any $s_*\in \R_+\setminus {\mathcal N}$ such that $s_*>s$.
If $f\in L^p(\R^d,\mu_s)$, then the function $g=G(s_*,s)f$ belongs to $L^p(\Rd, \mu_{s_*})$ and since $G(t,s)f=G(t,s_*)g$ and
$m_s(f)=m_{s_*}(g)$,
\begin{eqnarray*}
\lim_{t\to +\infty}\|G(t,s)f-m_s(f)\|_{L^p(\Rd,\mu_t)}=\lim_{t\to +\infty}\|G(t,s_*)g-m_{s_*}(g)\|_{L^p(\Rd,\mu_{t})}=0.
\end{eqnarray*}
Moreover, in view of \eqref{eq:contr}, it suffices to prove assertion for $f\in C^{\infty}_c(\R^d)$.
Finally, we can assume that $p>d$, since for $p\in [1,d]$ H\"older's inequality shows that
$\|G(t,s)f-m_s(f)\|_{L^p(\Rd,\mu_t)}\le \|G(t,s)f-m_s(f)\|_{L^{2d}(\Rd,\mu_t)}$ for all
$f\in C^{\infty}_c(\Rd)$ and $t>s$. Thus, we let $f\in C^{\infty}_c(\R^d)$ and $p>d$.

Fix a positive sequence $(t_n)$ diverging to $+\infty$, and
functions $\alpha_m$ in  $C_c^{\infty}(\R)$  such that
$0\le \alpha_m\le 1$ in $\Rd$ and $\alpha_m=1$  on $[-m,m]$ for each $m\in\N$.
We extend again $G(t,s)$ and $\mu_t$ to $\R$ as in Remark~\ref{rem:r}.
Proposition \ref{prop:gradT} implies that
\[
\int_{\R}\big\| \rho(s+t_n,\cdot) \alpha_m(s)^p \,|\nabla_xG(s+t_n,s)f|^p\big\|_{L^1(\Rd)}\,ds
 = \|\,|\nabla_x\cT(t_n)(\alpha_mf)|\, \|_{L^p(\R^{1+d},\nu)}^p
\]
tends to $0$ as $n\to+\infty$ for each $m\in\N$. There thus exist null sets ${\mathcal N}_m\subset [-m,m]$ and
subsequences $(t_n^{(m)})$ diverging to $+\infty$, with $t_k^{(m+1)}\in (t_n^{(m)})_n$ for all $k,m\in\N$, such that
\begin{eqnarray*}
\lim_{n\to +\infty}\int_{\Rd} \rho(s+t_n^{(m)},\cdot) \,|\nabla_xG(s+t_n^{(m)},s)f|^p \,dx=0
\end{eqnarray*}
for all $m\in \N$ and  $s\in\R_+\setminus {\mathcal N}_m$.
 We can thus determine a diagonal sequence $(t_{n_j})$ such that
\begin{equation}\label{eq:conv-grad-j}
\lim_{j\to +\infty}\int_{\Rd} \rho(s+t_{n_j},\cdot)|\nabla_xG(s+t_{n_j},s)f|^p \,dx
\end{equation}
for each $s\in\R_+\setminus{\mathcal N}$, where ${\mathcal N}=\bigcup_{m\in\N}{\mathcal N}_m$ is a null set.

Fix $s\in \R_+\setminus{\mathcal N}$. We use  Lemma~\ref{lem:rho} with $\tau=s+t_n$.
For every $k\in\N$,  it provides a number $\delta_k>0$ such that
 $\rho(s+t_n,x)\ge\delta_k$ for $n\in\N$  and $|x|\le k$. This lower bound and
 \eqref{eq:conv-grad-j} yield
 \begin{equation}\label{eq:conv-grad-k-j}
\lim_{j\to +\infty}\|\,|\nabla_xG(s+t_{n_j},s)f|\,\|_{L^p(B_k)}=0
\end{equation}
for each $k\in\N$. Observing that $\|G(s+t_{n_j},s)f\|_{L^p(B_k)}\le
c_k^{1/p}\|G(s+t_{n_j},s)f\|_{\infty}\le \|f\|_{\infty}$ for some positive constant $c_k$ (see Proposition~\ref{prop-cont-1}(i)),
we then find constants $\tilde c_k>0$ such that
$\|G(s+t_{n_j},s)f\|_{W^{1,p}(B_k)} \le \tilde c_k$
for all $j\in \N$. Since $p>d$, $W^{1,p}(B_k)$ is compactly embedded in $C(\overline{B_k})$. By a diagonal
argument, there exists a function $g(s,\cdot)\in C(\Rd)$ such that
$G(s+t_{n_j},s)f$ converges to $g(s,\cdot)$ locally uniformly in $\Rd$, up to a subsequence.
In particular, $\|g(s,\cdot)\|_{\infty} \le \|f\|_{\infty}$.

On the other hand, $\nabla_xG(s+t_{n_j},s)f$ tends to 0 in $L^p(B_R)^d$ as $n\to+\infty$, for every $R>0$,
due to \eqref{eq:conv-grad-k-j}. The weak gradient $\nabla_xg(s,\cdot)$ thus vanishes, and hence
$g(s)$ is constant in $x$. To prove that this constant is $m_s(f)$, it suffices to observe that
\begin{eqnarray*}
m_s(f)-g(s) = \int_{\Rd} (f -g(s))\,d\mu_{s}= \int_{\Rd} G(s+t_{n_j},s)(f -g(s))\,d\mu_{s+t_{n_j}}
\end{eqnarray*}
and use Lemma \ref{lem:conv}(i) with $\tilde\mu_n=\mu_{s+t_{n_j}}$. As a result, $g(s)=m_s(f)$ and
$G(s+t_{n_j},s)f$ tends to $m_s(f)$ locally uniformly
as $j\to+\infty$, for  $s\in \R_+\setminus{\mathcal N}$. Since $G(s+t_{n_j},s)m_s(f)\one= m_s(f)\one$
by  Proposition~\ref{prop-cont-1}(i), from Lemma~\ref{lem:conv}(i) we infer
\begin{eqnarray*}
\lim_{j\to+\infty}\|G(s+t_{n_j},s)(f-m_s(f))\|_{L^p(\Rd,\mu_{s+t_{n_j}})}=0.
\end{eqnarray*}
Finally, the function $h=\|G(\cdot,s)f-m_s(f)\|_{L^p(\Rd,\mu_t)}$ is decreasing in $[s,+\infty)$ since
\begin{align*}
h(t_2)&=\|G(t_2,s)(f-m_s(f))\|_{L^p(\Rd,\mu_{t_2})}
  =\|G(t_2,t_1)G(t_1,s)(f-m_s(f))\|_{L^p(\Rd,\mu_{t_2})}\\
&\le \|G(t_1,s)(f-m_s(f))\|_{L^p(\Rd,\mu_{t_1})}=h(t_1)
\end{align*}
for $s\le t_1<t_2$, where we have used property (i) in Proposition \ref{prop-cont-1} and \eqref{eq:contr}.
We conclude that $\lim_{t\to +\infty}\|G(t,s)f-m_s(f)\|_{L^p(\Rd,\mu_t)}=0$.

\smallskip
(ii) Fix $f\in C_b(\Rd)$, $s\in\R_+$, $R>0$ and $p>d$. Since $C_b(\Rd)\subset L^p(\Rd,\mu_s)$, $\|G(t+s,s)f-m_s(f)\|_{L^p(\Rd,\mu_{t+s})}$
tends to $0$ as $t\to +\infty$, by the first part of the proof.
Taking Lemma \ref{lem:rho} into account, we can estimate
\begin{eqnarray*}
\|G(t+s,s)f-m_s(f)\|_{L^p(B_R)}\le\delta_R^{-1/p}\|G(t+s,s)f-m_s(f)\|_{L^p(\Rd,\mu_{t+s})}
\end{eqnarray*}
for all $t\ge0$ and some positive constant $\delta_R$. Hence,
$\|G(t+s,s)f-m_s(f)\|_{L^p(B_R)}$ tends to $0$ as $t\to +\infty$. In particular, there exists a positive constant $C_1=C_1(R)$ such that
\begin{equation}
\|G(t+s,s)f-m_s(f)\|_{L^p(B_R)}\le C_1
\label{1}
\end{equation}
for all $t\ge0$. Moreover, the gradient estimate \eqref{eq:grad} implies that
\begin{align}
\|\nabla_xG(t+s,s)f\|_{L^p(B_R)}\le C_2\|f\|_{\infty}
\label{2}
\end{align}
for all $t\ge 1$ and some positive constant $C_2=C_2(R)$. From \eqref{1} and \eqref{2} we deduce that
the family of functions $\{G(t+s,s)f-m_s(f): t\ge 1\}$ is bounded in $W^{1,p}(B_R)$ and, consequently, in $C^{\beta}(B_R)$
for some $\beta\in (0,1)$ since $p>d$.
By the Arzel\`a-Ascoli theorem, from any sequence $(t_n)$ diverging to $+\infty$ we can extract a subsequence $(t_{n_k})$ such that
$G(t_{n_k}+s,s)f-m_s(f)$ converges uniformly in $B_R$ to zero as $k\to +\infty$, since it tends to zero in
$L^p(B_R)$. This shows that $G(t+s,s)f-m_s(f)$ tends to $0$, uniformly in $B_R$, as $t\to +\infty$.
\end{proof}

\begin{cor}\label{cor:unique}
$G(t,s)$ has exactly one tight evolution system of measures.
\end{cor}
\begin{proof}
Let  $\{\mu^{(1)}_t:t\ge0\}$  and $\{\mu^{(2)}_t:t\ge0\}$ be two evolution systems of measures
with corresponding means $m^{(i)}_t$. Fix $s\in\R_+$ and $f\in C_b(\Rd)$. Then, Theorem \ref{thm:main}(ii)
shows that, as $t\to +\infty$, $G(t,s)f$ converges both to $m_s^{(1)}(f)$ and $m_s^{(1)}(f)$, locally uniformly in $\Rd$.
Hence, $m_s^{(1)}(f)=m_s^{(2)}(f)$ and, consequently, $\mu_s^{(1)}=\mu_s^{(2)}$.
\end{proof}

\section{Converging coefficients}
In this section, we consider coefficients that converge as $t\to+\infty$ as described in the next
additional hypothesis.

\begin{hypo}
\label{ass:conv}
The coefficients $q_{i,j}$ and $b_i$ belong  to $C^{\alpha/2,\alpha}_b(\R_+\times B_R)$ for all $i,j\in \{1,\cdots,d\}$ and $R>0$
and $Q(t,\cdot)$ and $b(t,\cdot)$ converge pointwise to maps $Q_{\infty}:\R^d\to\R^{d^2}$ and
$b_{\infty}:\R^d\to\R^d$, respectively, as $t\to+\infty$.
\end{hypo}

\begin{remark}\label{rem:conv}
Hypotheses~\ref{hyp1} and \ref{ass:conv} imply that $Q_{\infty}\in C^\alpha_{\rm loc}(\R^d;\R^{d^2})$ and
$b_{\infty}\in C^\alpha_{\rm loc}(\R^d,\R^d)$ satisfy the $t$--independent analogues of  Hypotheses~\ref{hyp1}(ii) and (iii).
The evolution operator generated by $\cA_{\infty}$ is a semigroup $\{T(t):t\ge0\}$
which admits a single invariant measure $\mu_{\infty}$ having a
density  $\rho_{\infty}>0$ with respect to the Lebesgue measure.
(See e.g. Theorems~8.1.15  and 8.1.20 of \cite{BerLorbook} or \cite{MPW}.)
\end{remark}

As in Section \ref{sect:3}, $\{\mu_t: t\ge0\}$ is any tight evolution system of measures with densities $\rho(t,\cdot)$.
Under the additional Hypothesis \ref{ass:conv}, we  show that the densities $\rho(t,\cdot)$ converge to $\rho_\infty$
and we  derive a variant of Theorem~\ref{thm:main}.

\begin{prop}\label{prop:conv} The densities  $\rho(t,\cdot)$ converge to $\rho_{\infty}$ locally uniformly in $\Rd$ and
in $L^1(\Rd)$ as $t\to +\infty$.
\end{prop}

\begin{proof}
 We first prove local uniform convergence. It suffices to show that every sequence $(s_n)$ diverging to $+\infty$ admits
a subsequence such that $\rho(s_{n_j},\cdot)$ converges to $\rho_{\infty}$ locally uniformly on $\R^d$ as $j\to+\infty$.
As  in the proof of Theorem~6.2 of \cite{ALL} we see that
$\mu_t$ weakly$^*$ converges to $\mu_{\infty}$ as $t\to +\infty$.
Because of Proposition~\ref{prop:properties}(ii) and Hypothesis~\ref{hyp1}(i),
Corollary~3.9 \cite{BogKryRoc01} yields that
$\rho$ is contained in $C^{\beta}((s,s+1)\times B_R)$ for every $s,R>0$  and some $\beta>0$.
The proofs given there also yield that
the norms of $\rho$ in these spaces are bounded by a constant $C=C(R)$  independent of $s$.
See also  \cite{Krylov}. As a result, $\rho$ belongs to $C^{\beta}_b([0,+\infty)\times B_R)$ for every $R>0$.
The Arzel\`a-Ascoli theorem now provides a sequence $(t_n)$ diverging to $+\infty$
such that the density $\rho(t_n,\cdot)$ of the measure $\mu_{t_n}$ converges to
a function $g\in C(\Rd)$ locally uniformly in $\Rd$ as $n\to +\infty$.
The weak$^*$ convergence of $\mu_t$ to $\mu_{\infty}$ thus yields
\[
\int_{\R^d}f \rho_{\infty}\, dx=\int_{\Rd}f \,d\mu_{\infty} = \lim_{t_k\to +\infty}\int_{\R^d}f \,d\mu_{t_k}
  =\lim_{t_k\to +\infty}\int_{\R^d}f\rho(t_k,\cdot)\,dx   = \int_{\R^d}f g\,dx
\]
for every $f\in C^{\infty}_c(\R^d)$. Hence, $\rho_{\infty}= g$ and the local uniform convergence is shown.

 To prove the $L^1$-convergence, let $\ep>0$. By the tightness, there is a radius $R>0$ such that
$\mu_t(\Rd\setminus B_R),~\mu_{\infty}(\Rd\setminus B_R)\le \ep$ for all $t\ge0$.
>From the first part of the proof we deduce
\begin{align*}
\limsup_{t\to +\infty}\|\rho(t,\cdot)\!-\!\rho_{\infty}\|_{L^1(\R^d)}
&= \limsup_{t\to +\infty}\big[\|\rho(t,\cdot)\!-\!\rho_{\infty}\|_{L^1(B_R)}\!+\!\|\rho(t,\cdot)\!-\!\rho_{\infty}\|_{L^1(\Rd\setminus B_R)}\big]\\
&\le \limsup_{t\to +\infty}\mu_t(\Rd\setminus B_R)+\mu_{\infty}(\R^d\setminus B_R)\\
&\le 2\ep. \qedhere
\end{align*}
\end{proof}

\begin{theorem} \label{thm:conv}
Let $s\ge0$, $p\in[1,+\infty)$
and $f\in C_b(\Rd)$. Then, $G(t,s)f$ tends to $m_s(f)$ in $L^p(\Rd,\mu_{\infty})$
 as $t\to +\infty$.
\end{theorem}
\begin{proof}
The result follows from Proposition~\ref{prop:conv}, Theorem~\ref{thm:main}(i) and the estimates
\begin{align*}
\|G(t,s)&f-m_s(f)\|_{L^p(\Rd,\mu_{\infty})}^p \\
&\le \int_{\Rd} |\rho_{\infty}-\rho(t,\cdot)|\, |G(t,s)(f-m_s(f))|^p\,dx  + \|G(t,s)f-m_s(f)\|_{L^p(\Rd,\mu_t)}^p\\
&\le 2^p\|f\|_{\infty}^p \,\|\rho_{\infty}-\rho(t,\cdot)\|_{L^1(\Rd)} + \|G(t,s)f-m_s(f)\|_{L^p(\Rd,\mu_t)}^p. \qedhere
\end{align*}
\end{proof}

\section{An example}
\label{example}

We consider the family of operators $\cA(t)$ defined on smooth functions $\varphi$ by
\begin{eqnarray*}
\cA(t)\varphi=(1+|x|^2)^{\gamma}\sum_{i,j=1}^dq_{ij}^{(0)}(t,x)D_{ij}\varphi-b^{(0)}(t)(1+|x|^2)^r\langle x,\nabla_x\varphi\rangle
\end{eqnarray*}
for  $t\ge0$ and $x\in\Rd$, under the following assumptions.
\begin{enumerate}[\rm (i)]
\item
$q_{ij}^{(0)}=q_{ji}^{(0)}$  belong to $C^{\alpha/2,1+\alpha}_{\rm loc}(\R_+\times \R^d)\cap C_b(\R_+;C^1_b(\R^{d}))$
for some $\alpha\in(0,1)$ and for all $R>0$ and
$i,j\in\{1,\ldots,d\}$. Moreover, $\langle Q^{(0)}(t,x)\xi,\xi\rangle\ge\eta_0$ in $\R^{1+d}$ for some positive constant $\eta_0$ and every $\xi\in\partial B_1$;
\item
The function $b\in C^{\alpha/2}_{\rm loc}(\R_+)\cap  C_b(\R_+)$ satisfies $\beta:=\inf_{t\ge 0}b^{(0)}(t)>0$.
\item
$r>\gamma-1$ and $\gamma\in\R$.
\end{enumerate}
Let $\delta \in(0, 2(r+1-\gamma))$. Then every smooth and positive function $V:\R^d\to\R$ with
$V(x)=e^{|x |^{\delta}}$ for $x\in \R^d\setminus B_1$  satisfies Hypothesis~\ref{hyp1}(iii)
for the above operator. Indeed,  we have
\begin{align*}
(\mathcal{A}(t)V)(x)&=  \delta V(x)|x|^{\delta}\big[ (\delta |x|^{\delta -4}+ (\delta-2) |x|^{-4} )(1+|x|^2)^{\gamma}\langle Q^{(0)}(t,x)x,x\rangle\\
& \quad + \textrm{Tr}(Q^{(0)}(t,x))(1+|x|^2)^{\gamma}|x|^{-2}  -b^{(0)}(t)(1+|x|^2)^r\big]\\
&\le\delta V(x)|x|^\delta h(x),
\end{align*}
for $t\ge 0$ and $|x|\ge1$, where $h(x)= c\,|x|^{\delta-2} (1+|x|^2)^{\gamma}- \beta (1+|x|^2)^r$ tends to $-\infty$ as $|x|\to +\infty$
and $c>0$ is a constant depending on the bounds of $Q^{0}$.
One easily checks   the other conditions in Hypothesis \ref{hyp1}
 and the additional condition in Theorem \ref{thm:main}(ii).
Finally,  Hypothesis \ref{ass:conv} is satisfied
if $q_{ij}\in C^{\alpha/2,\alpha}_b(\R_+\times B_R)$ for all
$R>0$, $b\in C^{\alpha/2}(\R_+)$, and $Q^{(0)}(t,\cdot)$ and $b(t)$ converge to  $Q^{(0)}_{\infty}$ and $b_{\infty}$,
respectively, as $t\to +\infty$.


\begin{thebibliography}{99}

\bibitem{luciana}
L. Angiuli,
\newblock{\em Pointwise gradient estimates for evolution operators associated with Kolmogorov operators},
\newblock{Arch. Math. (Basel)} {\bf 101} (2013), 159-–170.



\bibitem{ALL}
L. Angiuli, L. Lorenzi, A. Lunardi,
\newblock{\em Hypercontractivity and asymptotic behaviour in nonautonomous Kolmogorov
equations}, \newblock{Comm.\ Partial Differential Equations {\bf 38} (2013)}, 2049--2080.


\bibitem{Ar}
D.G. Aronson.
\newblock{\em Non-negative solutions of linear parabolic equations},
\newblock{Ann.\ Scuola Norm.\ Sup.\ Pisa (3)} {\bf  22} (1968), 607--694.



\bibitem{aronson-besala67}
D.G. Aronson, P. Besala,
\newblock{\em Parabolic equations with unbounded coefficients},
\newblock{J. Differential Equations} {\bf 3} (1967), 1--14.

\bibitem{aronson-besala2}
D.G. Aronson and P. Besala
\newblock{\em Uniqueness of the positive solutions of parabolic equations with unbounded coefficients},
\newblock{Colloq. Math.} {\bf 18} (1967), 126--135.


\bibitem{BerLorbook}
M. Bertoldi, L. Lorenzi,
\newblock{Analytical Methods for Markov Semigoups},
\newblock{Chapman Hall/CRC Press, 2006}.


\bibitem{bodanko67}
W. Bodanko,
\newblock{\em Sur le probl\`eme de Cauchy et les probl\`emes de Fourier pour les \'equations paraboliques dans un domaine non born\'e},
\newblock{Ann. Polon. Math.} {\bf 18} (1966), 79--94.

\bibitem{BogKryRoc01}
V.I. Bogachev, N.V. Krylov, M. R\"ockner,
\newblock{\em On regularity of transition probabilities and invariant measures of singular
diffusion under minimal conditions},
\newblock{Comm.\ Partial Differential Equations} {\bf 26} (2001), 2037--2080.


\bibitem{CL}
C. Chicone, Y. Latushkin,
\newblock{Evolution Semigroups in Dynamical Systems and Differential Equations},
\newblock{American Mathematical Society,  1999}.

\bibitem{daprato-goldys}
G. Da Prato, B. Goldys,
\newblock{\em Elliptic operators on $\R^d$ with unbounded coefficients},
\newblock{J. Differential Equations} {\bf 172} (2001), 333--358.

\bibitem{daprato-lunardi}
G. Da Prato, A. Lunardi,
\newblock{\em On the Ornstein--Uhlenbeck operator in spaces of continuous functions},
\newblock{J. Funct. Anal.} (1995) {\bf 131}, 94--114.


\bibitem{DaPLun07Orn}
G. Da Prato, A. Lunardi,
\newblock{\em Ornstein-{U}hlenbeck operators with time periodic coefficients},
\newblock{J. Evol.\ Equ.} {\bf 7} (2007), 587--614.


\bibitem{DPR}
G. Da Prato, M. R\"ockner,
\newblock{\em A note on evolution systems of measures for time-dependent stochastic
differential equations}, \newblock{Proceedings of the 5th Seminar on Stochastic Analysis,
Random Fields and Applications, Ascona 2005},
\newblock{Progr.\ Probab.} {\bf 59}, Birkh\"auser Verlag, 2008, 115--122.

\bibitem{feller}
W. Feller,
\newblock{\em Diffusion processes in one dimension},
{Trans. Amer. Math. Soc.}, {\bf 77} (1954), 1--31.

\bibitem{friedlin}
M. Friedlin,
\newblock{Some remarks in the Smoluchowski-Kramers approximation},
{J. Stat. Physics},
{\bf 117} (2004), 617--634.

\bibitem{Friedman}
A. Friedman,
 \newblock{Partial Differential Equations of Parabolic Type,}
 \newblock{Prentice-Hall,  1964.}

\bibitem{GeisLun08}
M. Geissert,  A. Lunardi,
\newblock{\em Invariant measures and maximal $L^2$ regularity for nonautonomous
Ornstein-Uhlenbeck equations},
\newblock{J. Lond.\ Math.\ Soc.\ (2)}  {\bf{77}}  (2008), 719--740.

\bibitem{GeisLun09}
M. Geissert,  A. Lunardi,
\newblock{\em Asymptotic behavior and hypercontractivity in nonautonomous
Ornstein-Uhlenbeck equations},
\newblock{J. Lond.\ Math.\ Soc.\ (2)} {\bf{79}} (2009),  85--106.

\bibitem{Krylov}
N.V. Krylov,
\newblock{\em Some properties of traces for stochastic and deterministic parabolic weighted
Sobolev spaces},
\newblock{J. Funct.\ Anal.} {\bf 183} (2001), 1--41.

\bibitem{krzyzanski62-a}
M. Krzy\.za\'nski,
\newblock{\em Sur la solution fondamentale de l'\'equation lin\'eaire normale du type
parabolique dont le dernier coefficient est non born\'e. I},
\newblock{Atti Accad. Naz. Lincei Rend. Cl. Sci. Fis. Mat. Nat. {\rm (}8{\rm )}} {\bf 32} (1962),
326--330.

\bibitem{krzyzanski62-b}
M. Krzy\.za\'nski,
\newblock{\em Sur la solution fondamentale de l'\'equation lin\'eaire normale du type
parabolique dont le dernier coefficient est non born\'e. II},
\newblock{Atti Accad. Naz. Lincei Rend. Cl. Sci. Fis. Mat. Nat. {\rm (}8{\rm )}} {\bf 32} (1962),
471--476.

\bibitem{KunLorLun09Non}
M. Kunze, L. Lorenzi, A. Lunardi,
\newblock {\em Nonautonomous Kolmogorov parabolic equations
with unbounded coefficients},
\newblock {Trans.\ Amer.\ Math.\ Soc.} {\bf{362}} (2010), 169--198.



\bibitem{LorLunZam10}
L. Lorenzi, A. Lunardi, A. Zamboni,
\newblock{\em Asymptotic behavior in
time periodic parabolic problems with unbounded coefficients},
\newblock{J. Differential Equations} {\bf{249}} (2010), 3377--3418.

\bibitem{LorZam09}
L. Lorenzi, A. Zamboni,
\newblock{\em Cores for parabolic operators with unbounded coefficients},
\newblock{J. Differential Equations} {\bf{246}} (2009), 2724--2761.

\bibitem{MPW}
G. Metafune, D. Pallara, M. Wacker,
\newblock {\em Feller semigroups on $\R^N$},
\newblock {Semigroup Forum} {\bf 65} (2002), 159--205.

\bibitem{PRS}
J. Pr\"uss, A. Rhandi, R. Schnaubelt,
\newblock{\em The domain of elliptic operators on $L^p(\Rd)$ with unbounded drift coefficients},
\newblock{Houston J. Math.} {\bf 32} (2006), 563--576.
\end{thebibliography}
\end{document}